\documentclass[12pt, reqno, a4paper]{amsart}

\usepackage{ amssymb, amsmath, amsfonts, amsthm, mathrsfs, url, bm, mathtools, enumitem}

\usepackage{xcolor}  	
\usepackage{hyperref}

\hypersetup{
colorlinks,
   linkcolor={cyan!80!black},
   citecolor={cyan!80!black},
 urlcolor={cyan!80!black}
}

\usepackage{color}

\usepackage[margin=1in]{geometry}

\RequirePackage{doi}

\usepackage{amscd}
\usepackage{amsfonts}
\usepackage{float}
\usepackage{color}
\usepackage[
backend=biber,
style=alphabetic,
]{biblatex}
\usepackage{bookmark}

\renewbibmacro{in:}{}
\DeclareFieldFormat{title}{#1}

\DeclareFieldFormat[article]{title}{\mkbibemph{#1}}       
\DeclareFieldFormat[incollection]{title}{\mkbibemph{#1}}  
\DeclareFieldFormat[book]{title}{\mkbibemph{#1}}          
\DeclareFieldFormat[incollection]{booktitle}{#1}          
\DeclareFieldFormat[article]{journaltitle}{#1}            

\AtEveryBibitem{%
  \ifentrytype{misc}{\DeclareFieldFormat{title}{\mkbibemph{#1}}}{}}

\DeclareFieldFormat{eprint:eprint}{arXiv:\href{https://arxiv.org/abs/#1}{#1}}

\DeclareFieldFormat[inproceedings]{title}{\mkbibemph{#1}}

\DeclareFieldFormat[inproceedings]{booktitle}{#1}

\usepackage{amssymb}

\newtheorem{theorem}{Theorem}[section]
\newtheorem{lemma}{Lemma}[section]

\newtheorem{proposition}{Proposition}[section]

\theoremstyle{definition}
\newtheorem{definition}{Definition}[section]

\theoremstyle{remark}

\numberwithin{equation}{section}

\newcommand{\Mod}[1]{\ (\mathrm{mod}\ #1)}

\renewcommand{\Re}{\mathrm{Re}}

\newcommand{\N}{{\mathbb N}}
\newcommand{\E}{{\mathbb E}}

\newcommand{\T}{\mathbb T}

\newcommand{\D}{\mathbb D}

\renewcommand{\leq}{\leqslant}
\renewcommand{\geq}{\geqslant}

\begin{document}

\title[Multiplicative functions along linear fractional sequences]{Value distribution of multiplicative functions along linear fractional sequences}


\author{Sun-Kai Leung}
\address{Mathematical Institute, University of Oxford, Andrew Wiles Building\\ Radcliffe Observatory Quarter, Woodstock Rd\\
Oxford OX2 6GG\\
United Kingdom}
\email{sunkaileung@gmail.com}


\subjclass[2020]{11J71, 11N64}

\date{}

\dedicatory{}

\keywords{}

\begin{abstract}
For $a,c\in\mathbb{N}$ and $b,d\in\mathbb{Z}$ such that the (non-empty) set
\[
R_{a,b,c,d}
:=\left\{\frac{an+b}{\,cn+d\,}: n\in\mathbb{N}\right\}
\cap\bigl(\mathbb{Q}_{>0}\setminus\{1\}\bigr)
\]
is multiplicatively recurrent, we give a complete characterization of
the set of limit points of every unimodular multiplicative function
$f\in\mathcal{M}$ along $R_{a,b,c,d}.$ We show that the possible limit sets are either the finite subgroups of the unit circle or the entire circle, thereby extending the dichotomy of Klurman--Mangerel.
\end{abstract}

\maketitle

\section{Introduction}
Denote the set of unimodular completely multiplicative functions by
\begin{align*}
\mathcal{M}:=\{ f: \mathbb{Q}_{>0} \to \mathbb{T} \,:\, f(rs)=f(r)f(s) \text{ for $r,s \in \mathbb{Q}_{>0}$} \},
\end{align*}
which can be identified with the Pontryagin dual of the group of positive rationals $(\mathbb{Q}_{>0},\times).$
By the pigeonhole principle, for every $f \in \mathcal{M}$ and $\epsilon>0$,
there exist infinitely many rational numbers $r \in \mathbb{Q}_{>0}$ such that
\[
\left| f( r ) - 1 \right| < \epsilon,
\]
which can be viewed as the multiplicative analogue of Dirichlet's
approximation theorem. It is then natural to restrict this 
approximation problem to a thin subset of \(\mathbb{Q}_{>0}.\)

In this direction, Klurman and Mangerel \cite{MR3856825} showed that for every $f \in \mathcal{M},$ we have
\begin{align} \label{eq:n+1n}
\liminf_{n \to \infty} \left| f\left( \frac{n+1}{n} \right) -1 \right|=0.
\end{align}
Let $a,c\in\mathbb{N}$ and $b,d\in\mathbb{Z}.$
As a generalization, Charamaras, Mountakis, and Tsinas \cite{charamaras2024multiplicativerecurrencelinearpatterns} showed that 
\begin{align} \label{eq:cmt}
\liminf_{n \to \infty} \left| f\left( \frac{an+b}{cn+d} \right) -1 \right|=0
\end{align}
for every $f \in \mathcal{M}$ if and only if $a=c$ and either $b=d$ or  $a \mid \operatorname{lcm}(b,d).$\footnote{Their result is stated under the normalization
\((a,b,c,d)=1\); the formulation above is equivalent after dividing
\(a,b,c,d\) by their greatest common divisor.} See \cite[Corollary 1.7]{MR4594405} for the case where $a=c$ and either $a \mid b$ or $a \mid d.$ 

For \( a, c \in \mathbb{N} \) and \( b, d \in \mathbb{Z} \), define
\[
R_{a,b,c,d} := \left\{ \frac{an + b}{cn + d} \,:\, n \in \mathbb{N} \right\} \cap (\mathbb{Q}_{>0} \setminus \{1\}).\]
Recently, Táfula and the author~\cite{leung2025multiplicativerecurrencemobiustransformations}
proved that the set $R_{a,b,c,d}$, if non-empty, is multiplicatively recurrent
(see, for instance, \cite{leung2025multiplicativerecurrencemobiustransformations} for the definition)
if and only if $a=c$ and $a \mid \operatorname{lcm}(b,d)$.
In particular, for every $f_1, \ldots, f_r \in \mathcal{M},$ we have
\begin{align*}
\liminf_{n \to \infty} \max_{1 \leq j \leq r} \left|f_j \left(\frac{an+b}{an+d} \right)-1\right| = 0,
\end{align*}
thereby extending the aforementioned result of Charamaras, Mountakis, and Tsinas to simultaneous approximation.

For a sequence of complex numbers $(z_n)$, we denote its set of limit points by
\begin{align*}
 \omega(z_n):=\bigcap_{N=1}^{\infty} \overline{\{ z_n \,:\, n \geq N \}}.
\end{align*}
Also, for $k \in \mathbb{N},$ we denote the group of $k$-th roots of unity by
\begin{align*}
\mu_k:=\{ z \in \mathbb{T} \,:\, z^k=1 \}.
\end{align*}
More precisely, Klurman and Mangerel \cite{MR3856825} showed that
\begin{align*}
  \omega \left(  f \left( \frac{n+1}{n} \right)  \right)=
\begin{cases}
\mu_k  & \mbox{{\normalfont if $f \cdot n^{-it} $ has order $k \in \mathbb{N}$ for some $t \in \mathbb{R},$}} \\
\mathbb{T} & \mbox{{\normalfont otherwise, }}
\end{cases}
\end{align*}
where the \textit{order} of $g \in \mathcal{M}$ is the minimum $k \in \mathbb{N}$ for which $g^k \equiv 1.$
In particular, we have
\[
1\in
\omega\left(
f\left(\frac{n+1}{n}\right)
\right)
\]
for every \(f\in\mathcal{M}\), which is equivalent to
\eqref{eq:n+1n}.

In this paper, we extend their dichotomy to those integers $a\in\mathbb{N}$ and $b,d\in\mathbb{Z}$ satisfying $a\mid\operatorname{lcm}(b,d)$, thereby
strengthening \eqref{eq:cmt} as well.\footnote{Without loss of generality, we assume $b , d \geq 0$. If $b=d,$ then the set of limit points is $\{1\}.$}

\begin{theorem} \label{thm:main}
Let \(a\geq 1\) and \(b,d\geq 0\) be integers for which
\(b\neq d\) and \(a\mid\operatorname{lcm}(b,d)\). Then for every $f \in \mathcal{M},$ we have
\begin{align*}
\omega \left(  f \left( \frac{an+b}{an+d} \right)  \right)=
\begin{cases}
\mu_k  & \mbox{{\normalfont if $f \cdot n^{-it} $ has order $k \in \mathbb{N}$ for some $t \in \mathbb{R},$}} \footnotemark \\
\mathbb{T} & \mbox{{\normalfont otherwise. }}
\end{cases}
\end{align*}
\end{theorem} 
\footnotetext{Such \(t\) and \(k\), if they exist, are uniquely determined.} 

\noindent\textit{Notation.} 
Throughout the paper, we use the standard big $O$ and little $o$ notations, as well as the Vinogradov notation $\ll, \gg,$ where the implied constants depend only on the subscripted parameters. We denote $e(x):=e^{2\pi i x}$ for $x \in \mathbb{R}.$ 
Given a non-empty finite subset $S \subseteq \mathbb{N}$ and a function (or distribution) $a: S \to \mathbb{C},$ we write 
\begin{align*}
\mathbb{E}_{n \in S} \, a(n):=\frac{1}{|S|}\sum_{n \in S} a(n),
\end{align*}
and
\begin{align*}
\mathbb{E}^{\log}_{n \in S} a(n):=\left( \sum_{n \in S} \frac{1}{n} \right)^{-1}\sum_{n \in S} \frac{a(n)}{n}.
\end{align*}

\section{Pretentious number theory}

\textit{Pretentious number theory} (see \cite{gs} for an introduction) is instrumental in Tao's resolution of the Erdős discrepancy problem \cite{MR3533300}, as well as to the partition regularity of (generalized) Pythagorean pairs established by Frantzikinakis, Klurman, and Moreira \cite{frantzikinakis2023partitionregularitypythagoreanpairs,frantzikinakis2024partitionregularitygeneralizedpythagorean}.  We begin by recalling the following definitions.

\begin{definition}[Pretentious distance] For $f,g \in \mathcal{M},$ and $0<x<y \leq \infty,$ the squared \textit{pretentious distance} between $f$ and $g$ on $(x,y]$ is defined as
\begin{align*}
\mathbb{D}(f,g;x,y)^2:=&\sum_{x<p \leq y} \frac{1-\Re(f(p)\overline{g(p)})}{p}  
=\frac{1}{2} \sum_{x<p \leq y} \frac{|f(p)-g(p)|^2}{p}.
\end{align*}
If $\mathbb{D}(f,g;1,\infty)< \infty,$ we say that $f$ \textit{pretends} to be $g,$ and denote this by $f \sim g.$\footnote{The ordered pair $(\mathcal{M},\mathbb{D})$ forms an extended metric space, and $\sim$ is an equivalence relation.}
\end{definition}

\begin{definition}[Positively modified Dirichlet character]
Let $\chi$ be a Dirichlet character. Its \textit{positively modified Dirichlet character} $\widetilde{\chi}: \mathbb{N} \to \mathbb{T}$ is defined on primes by
\begin{align*}
\widetilde{\chi}(p)=
\begin{cases}
\chi(p) & \mbox{{\normalfont if $\chi(p) \neq 0,$ }} \\
\hfil  1 & \mbox{{\normalfont if $\chi(p) = 0, $ } } 
\end{cases}
\end{align*}
and extended completely multiplicatively.
\end{definition}

\begin{definition}[Pretentious function]
A function $f \in \mathcal{M}$ is \textit{pretentious} if there exist a primitive Dirichlet character $\chi$ and $t \in \mathbb{R}$ such that $f \sim \widetilde{\chi} \cdot n^{it}.$\footnote{The primitive Dirichlet character $\chi$ and $t \in \mathbb{R}$ are uniquely determined.} Otherwise, it is \textit{non-pretentious}.\footnote{It is sometimes called \textit{aperiodic}.}
\end{definition}

\begin{definition}[Strongly non-pretentious function]
A function $f \in \mathcal{M}$ is \textit{strongly non-pretentious} if, for every $R \geq 1,$ we have
\begin{align*}
\lim_{N \to \infty}  \inf_{\substack{ \chi \Mod{r},\, r\leq R\\
                  |t|\leq RN}}
 \D(f,\widetilde{\chi} \cdot n^{it};1,N) = \infty.
\end{align*}
\end{definition}

Many non‑pretentious functions fail to be strongly non‑pretentious. Nevertheless, both non‑pretentious functions and their powers admit an upgrade to strong non‑pretentiousness along a common subsequence.

\begin{lemma}[Simultaneous strong non-pretentiousness on a subsequence]
\label{lem:simultaneous-nonpretentiousness}

Let $\mathcal E\subseteq\N$ be a finite non-empty set. Suppose $f^h$ is
non-pretentious for every $h\in\mathcal E$. Then for every $R \geq 1$, there
exists a strictly increasing sequence of positive integers $(N_j)$ such that
\[ \lim_{j \to \infty}
 \min_{h\in\mathcal E}
 \inf_{\substack{ \chi \Mod{r},\, r\leq R\\
                  |t|\leq RN_j}}
 \D(f^h,\widetilde{{\chi}} \cdot n^{it};1,N_j)
=\infty.
\]

\begin{proof}
Let $H:=\prod_{h \in \mathcal E} h.$ If $f^H$ is pretentious, then  $f^h$ is strongly non-prententious for every $h \in \mathcal{E}$ (see \cite[Corollary 2.7]{charamaras2024multiplicativerecurrencelinearpatterns}). Otherwise, if $f^H$ is non-pretentious, then there exists $(N_j)$ such that (see \cite[Corollary 2.11]{charamaras2024multiplicativerecurrencelinearpatterns} for instance)
\begin{align} \label{eq:stnonpret}
 \lim_{j \to \infty}
 \inf_{\substack{ \psi \Mod{r},\, r\leq HR\\
                  |u|\leq (HR)N_j}}
 \D(f^H,\widetilde{{\psi}} \cdot n^{iu};1,N_j)
=\infty.
\end{align}
For $h \in \mathcal{E},$ let $s_h:=H/h.$ Then for every $\chi \Mod{r}$ with $r \leq R$ and $|t| \leq RN_j,$ Minkowski's inequality gives
\begin{align*}
\D (f^H, \widetilde{\chi}^{s_h}n^{s_h i t};1,N_j) \leq s_h \cdot  \D(f^h,\widetilde{\chi} \cdot n^{i t};1,N_j ),
\end{align*}
and therefore
\begin{align*}
 \D(f^h,\widetilde{\chi} \cdot n^{i t};1,N_j ) \geq \frac{h}{H} \cdot \D (f^H, \widetilde{\chi}^{s_h}n^{s_h i t};1,N_j).
\end{align*}
Since $|s_h t| \leq H(RN_j)$ and $\mathcal E$ is finite, it follows from (\ref{eq:stnonpret}) that
\begin{align*}
 \lim_{j \to \infty} \min_{h \in \mathcal{E}}
 \inf_{\substack{ \chi \Mod{r},\, r\leq R\\
                  |t|\leq RN_j}}
 \D(f^h,\widetilde{{\chi}} \cdot n^{it};1,N_j)
=\infty,
\end{align*}
and the proof is complete.
\end{proof}
\end{lemma} 

\begin{definition}[Multiplicative F\o{}lner sequence]

A sequence $(\Phi_K)$ of non-empty finite subsets of $\mathbb{Q}_{>0}$ is a \textit{multiplicative F\o{}lner sequence} if for every $r \in \mathbb{Q}_{>0},$ we have
\begin{align*}
\lim_{K \to \infty} \frac{|(r \cdot \Phi_K) \cap \Phi_K|}{|\Phi_K|} =1.
\end{align*}
\end{definition}
For instance, one can verify that the sequence $(\Phi_K)$ defined by
\begin{align*}
\Phi_K:= \left\{ \prod_{p \leq K} p^{v_p} \,:\, K<v_p \leq 2K \right\}
\qquad \text{for $K \in \mathbb{N}$}
\end{align*}
is a multiplicative F\o{}lner sequence (see \cite[Proposition 2.5]{frantzikinakis2023partitionregularitypythagoreanpairs}). Throughout the paper, this is the only multiplicative F\o{}lner sequence to be considered.

Using the multiplicative F\o{}lner sequence $(\Phi_K)$, we obtain an orthogonality relation.

\begin{lemma}[Multiplicative orthogonality] \label{lem:orthog}
Let $f \in \mathcal{M}.$ Then
\begin{align*}
\lim_{K \to \infty} \mathbb{E}_{Q \in \Phi_{K}} f(Q) =
\begin{cases}
1 & \mbox{{\normalfont if $f \equiv 1,$ }} \\
\hfil  0 & \mbox{{\normalfont otherwise.} } 
\end{cases}
\end{align*}
Moreover, for any prime $p \leq K$ with $f(p) \neq 1,$ we have
\begin{align*}
| \mathbb{E}_{Q \in \Phi_{K}} f(Q) | \ll \frac{1}{K|1-f(p)|}.
\end{align*}

\begin{proof}
This is essentially \cite[Lemma 3.2]{frantzikinakis2023partitionregularitypythagoreanpairs}. It suffices to assume $f \not\equiv 1.$ By definition, we have
\begin{align*}
 \mathbb{E}_{Q \in \Phi_{K}} f(Q) = \prod_{p \leq K} \left(\frac{1}{K} \sum_{K < v_p \leq 2K} f(p)^{v_p} \right).
\end{align*}
If $f(p) \neq 1,$ then
\begin{align*}
\sum_{K < v_p \leq 2K} f(p)^{v_p} = f(p)^{K+1} \cdot \frac{1-f(p)^{K}}{1-f(p)}, 
\end{align*}
which implies
\begin{align*}
 |\mathbb{E}_{Q \in \Phi_{K}} f(Q)| \leq 
 \prod_{\substack{p \leq K\\f(p) \neq 1}} \frac{2}{K|1-f(p)|},
\end{align*}
and the lemma follows.
\end{proof}
\end{lemma}

For $n \in \mathbb{N},$ by writing $f(n)=e(g(n))$ for some suitable additive function $g:\mathbb{N} \to \mathbb{R},$ the following is a consequence of the Turán--Kubilius inequality.

\begin{lemma}[Concentration inequality] \label{lem:concentration}
Let $f \in \mathcal{M},$ and $\chi$ be a Dirichlet character of modulus $q.$ Let $K,Q \geq 1$ be integers for which 
\begin{align*}
 \prod_{p \leq K} p \mid Q \qquad \text{\rm and} \qquad q \mid Q,
\end{align*}
and all prime divisors of $Q$ are at most $K$. 
 If $N > K$, then for any positive integer $(a,Q)=1$ and $t \in \mathbb{R},$ we have
\begin{gather*}
\E_{n\leq N}^{\log}
\left|
f(Qn+a)-\chi(a)(Qn)^{it}
\exp(-F(\chi,t;K,N))
\right|
\\
\ll
\D(f,\widetilde\chi \cdot n^{it};K,\infty)
+K^{-1/2}+o_{Q,a,t;N \to \infty}(1),
\end{gather*}
where
\begin{align*}
F(\chi,t;K,N):=\sum_{\substack{K<p \leq N}} \frac{1-f(p)\overline{\chi}(p)p^{-it}}{p}.
\end{align*}

\begin{proof}
This is a logarithmically weighted version of \cite[Lemma 2.5]{MR4328688} (see also \cite[Lemma 2.13]{charamaras2024multiplicativerecurrencelinearpatterns}).
\end{proof}

\end{lemma}

For pretentious functions, the concentration inequality shows that their (logarithmic) auto‑correlation stabilizes.

\begin{lemma}[Pretentious logarithmic two-point correlation] \label{lem:pretcorr}
Let $f \in \mathcal{M},$ and $\chi$ be a Dirichlet character of modulus $q.$ Let $K,Q_1, Q_2 \geq 1$ be integers for which 
\begin{align*}
 \prod_{p \leq K} p \mid  \operatorname{gcd}(Q_1,Q_2), \qquad q \mid \operatorname{gcd}(Q_1,Q_2), 
\end{align*}
and all prime divisors of $Q_1, Q_2$ are at most $K$. If $N > K$, then for any positive integers $(a_1,Q_1)=1, (a_2,Q_2)=1$ and $t \in \mathbb{R},$ we have
\begin{gather*}
\E_{n\leq N}^{\log}
f(Q_1n+a_1)\overline{f(Q_2n+a_2)} 
=\chi(a_1)\overline{\chi(a_2)}(Q_1/Q_2)^{it}
\exp(-2\Re F(\chi,t;K,N)) \\
+O(
\D(f,\widetilde\chi \cdot n^{it};K,\infty)
+K^{-1/2})++o_{Q_1,Q_2,a_1,a_2,t;N \to \infty}(1).
\end{gather*}
\begin{proof}
Using the inequality $|z_1 \overline{z_2}-w_1\overline{w_2}| \leq |z_1-w_1|+|z_2-w_2|$ for $z_1,z_2,w_1,w_2 \in \mathbb{T},$ the lemma follows immediately from Lemma \ref{lem:concentration}.
\end{proof}

\end{lemma}

For (strongly) non-pretentious functions, Tao's estimate shows that their (logarithmic) auto-correlation diminishes.

\begin{lemma}[Non-pretentious logarithmic two-point correlation]
\label{lem:two-point-correlation}
Let $a,b,c,d\in\N$ satisfy
\[
 ad-bc\neq0.
\]
For every $\epsilon>0$, there exists
$R=R(\epsilon,a,b,c,d)\geq1$ such that the following
holds. Let $f \in \mathcal{M}$ and $N> R.$ If
\[
 \inf_{\substack{ \chi \Mod{r},\, r\leq R\\
                  |t|\leq RN}}
 \D(f,\widetilde{\chi} \cdot n^{it};1,N)\geq R,
\]
then
\[
 \left|
 \E_{n\leq N}^{\log}
 f(an+b)\overline{f(cn+d)}
 \right|\leq\epsilon.
\]

\begin{proof}
See \cite[Theorem 1.3]{MR3569059}.
\end{proof}

\end{lemma}

\section{Preparation}

To prove Theorem \ref{thm:main}, we follow the same underlying strategy as in \cite{charamaras2024multiplicativerecurrencelinearpatterns}. However, by starting with the reduction processes introduced in \cite{leung2025multiplicativerecurrencemobiustransformations}, it suffices to verify certain special cases, leading to a streamlined argument.

\begin{lemma}[Arithmetic normalization] \label{lem:normalization}
Let $a \geq 1$ and $b,d \geq 0$ be integers satisfying $b \neq d$ and $a \mid \operatorname{lcm}(b,d).$ Set 
\[
 g:=\operatorname{gcd}(a,b,d),
 \quad
 a':=a/g,
 \quad
 b':=b/g,
 \quad
 d':=d/g.
\]
Then $\operatorname{gcd}(a', b'-d')=1$ and $a' \mid b' d'.$
\begin{proof}
Without loss of generality, assume that \(b>d\). If \(a=1\) or \(d=0\),
then both assertions are immediate. Hence assume that \(a\geq 2\) and
\(b,d>0\).

Since $\operatorname{lcm}(b,d) = g\operatorname{lcm}(b',d'),$ the assumption \(a\mid\operatorname{lcm}(b,d)\) implies that $a'\mid\operatorname{lcm}(b',d').$
Now let \(p\) be any prime divisor of \(a'\). Since \(a'\mid b'd'\), we
have \(p\mid b'd'\). On the other hand,
\(\gcd(a',b',d')=1\), so \(p\) cannot divide both \(b'\) and \(d'\).
Thus \(p\) divides exactly one of \(b'\) and \(d'\), and consequently
\(p\nmid b'-d'\). Since this holds for every prime divisor \(p\) of
\(a'\), we conclude that $\gcd(a',b'-d')=1.$
\end{proof}

\end{lemma}

\begin{lemma}[Case reduction]  \label{lem:reduction}
 Let $a \geq 1$ and $b, d \geq 0$ be integers. Suppose $b>d, \, \gcd(a,b-d)=1,$ and $a \mid bd.$ Then there exist $A,B\in\N$ with $B\geq 2$ and $A=B(B-1)$, and
$C,D\in\N$ such that, for every integer $m\geq 1$, if $n=Cm+D$, then
\[
 \frac{an+b}{an+d}
 =\frac{Am+B}{Am+B-1}.
\]
\end{lemma}

\begin{proof}
See \cite[Lemma 3.2]{leung2025multiplicativerecurrencemobiustransformations}.
\end{proof}

Therefore, it suffices to prove Theorem \ref{thm:main} in the above cases, further restricted to a sub‑progression chosen so that, after invoking the Erdős--Turán inequality or Fourier inversion on finite cyclic groups, both Lemma \ref{lem:pretcorr} and Lemma \ref{lem:two-point-correlation} become applicable.

\begin{lemma}[Restriction to a sub-progression]
\label{lem:bezout-progression}
Let $q \geq 1$ and $B \geq 2$ be integers. Then there exists $K(q,B)\geq 1$ such that for every integer $K \geq K(q,B),$ the following holds:
\begin{enumerate}[label={\rm (\alph*)}]
    \item $qB(B-1) \mid \prod_{p \leq K} p^K;$ 
    \item for every $Q = \prod_{p \leq K} p^{v_p} \in \Phi_K,$ denote 
    \begin{align*}
    Q^+:=\prod_{\substack{p \leq K\\ p \mid B}} p^{v_p}, \qquad 
    Q^-:=\prod_{\substack{p \leq K\\ p \nmid B}} p^{v_p},
    \end{align*}
    and 
    \begin{align*}
       L:=Q^2, \qquad R^+:=(B-1)Q^-Q, \qquad R^-:=BQ^+Q.
    \end{align*}
    There exist integers $0 \leq \ell <L, \, 0 \leq r^{\pm} <R^{\pm}$ such that for every integer $n \geq 0,$ we have
    \begin{align*}
    (B-1)(Ln+\ell)+1=Q^{+}(R^+n+r^+),
    \end{align*}
    and
    \begin{align*}
            B(Ln+\ell)+1=Q^-(R^-n+r^-);
    \end{align*}
    \item $r^+ \equiv r^- \Mod{Q}$ and $\operatorname{gcd}(r^{\pm},R^{\pm})=1;$
    \item $R^+r^{-}-R^-r^+=-Q \neq 0.$
\end{enumerate}

\begin{proof}
Let $Q \in \Phi_K$ and set 
\[
 S^+:=(B-1)Q^-,
 \qquad
 S^-:=BQ^+.
\]
Since every prime $p\leq K$ divides exactly one of $S^+$ and $S^-,$ we have
\begin{equation}
 (S^+,S^-)=1,
 \qquad
 (S^+-S^-,Q)=1.
 \label{eq:R-coprime}
\end{equation}
Let $(r_0^+,r_0^-)$ be an integral solution to the Bézout equation 
\begin{align} \label{eq:bezout}
S^- r^+ -S^+ r^- =1. 
\end{align}
Then all integral solutions $(r^+, r^-)$ are given by 
\begin{align*}
r^+=S^+ k+r_0^+, \qquad r^-=S^-k+r_0^- \qquad \text{for $k \in \mathbb{Z}.$}
\end{align*}
In particular, the congruence 
\begin{align} \label{eq:r+r-Q}
r^+ \equiv r^- \Mod{Q}
\end{align}
is equivalent to
\begin{align*}
(S^+ - S^-) k + (r_0^+ - r_0^-) \equiv 0 \Mod{Q},
\end{align*}
which is solvable by (\ref{eq:R-coprime}). Replacing $k$ by $k+Qj$ for some integer $j \in \mathbb{Z}$
preserves this congruence and changes $r^+$ by
$S^+Qj=(B-1)Q^-Qj=R^+j.$ We may therefore arrange
\begin{equation*}
 0 \leq r^+< R^+,
\end{equation*}
and 
\begin{align*}
0 \leq r^-= \frac{S^-r^+-1}{S^+} < \frac{S^- R^+}{S^+} =R^-.
\end{align*}
Also, by (\ref{eq:bezout}) and (\ref{eq:r+r-Q}), we have 
\begin{align*}
\operatorname{gcd}(r^{\pm}, R^{\pm})=1.
\end{align*}
Reducing (\ref{eq:bezout}) modulo $B-1$ gives
\begin{align*}
BQ^+ r^+ \equiv Q^+r^+ \equiv  1 \Mod{B-1}.
\end{align*}
Therefore, $Q^+r^+-1$ is divisible by $B-1,$ and we define
\begin{align} \label{eq:ell}
\ell:=\frac{Q^+r^+-1}{B-1}.
\end{align}
Since $Q^+r^+ < Q^+ R^+ = (B-1)Q^2,$ we have 
\begin{align*}
0 \leq \ell <Q^2=L.
\end{align*}
Using (\ref{eq:bezout}) and (\ref{eq:ell}), we obtain
\begin{align} \label{eq:ell2}
\ell=\frac{Q^-r^- -1}{B}.
\end{align}
Since $Q^+R^+=(B-1)Q^2=(B-1)L,$ it follows from (\ref{eq:ell}) that
   \begin{align*}
    (B-1)(Ln+\ell)+1=Q^{+}(R^+n+r^+)
    \end{align*}
    for every integer $n \geq 0.$ Similarly, since $Q^-R^-=BQ^2=BL,$ it follows from (\ref{eq:ell2}) that
    \begin{align*}
        B(Ln+\ell)+1=Q^-(R^-n+r^-)
    \end{align*}
    for every integer $n \geq 0.$ Finally, the expression (\ref{eq:bezout}) gives
    \begin{align*}
    R^+r^{-}-R^-r^+ = Q(S^+r^- - S^-r^+) = -Q,
    \end{align*}
    and the lemma follows.
\end{proof}

\end{lemma}

\section{Fourier Flatness}

Let $f \in \mathcal{M},$ $B \geq 2$ be an integer and $A=B(B-1).$ 
Given $K \in \mathbb{N}$ and $Q \in \Phi_K,$ let $L_Q$ and $\ell_Q$ be furnished by Lemma \ref{lem:bezout-progression}, and define
\begin{align*}
z_{Q,n}:=f \left( \frac{A(L_Q n + \ell_Q)+B}{A(L_Qn  + \ell_Q)+(B-1)} \right)
\qquad \text{for $n \in \mathbb{N},$}
\end{align*}
and
\begin{align*}
{\sigma}_{K,N}:=\mathbb{E}_{Q \in \Phi_K} 
\mathbb{E}^{\log}_{n \leq N} \delta_{z_{Q,n}} 
\qquad \text{for $N \in \mathbb{N}.$}
\end{align*}
Also, we denote its Fourier coefficients by
\begin{align*}
 \widehat\sigma_{K,N}(h):=\int_\T z^h\,d\sigma_{K,N}(z) =
 \mathbb{E}_{Q \in \Phi_K} 
\mathbb{E}^{\log}_{n \leq N} z_{Q,n}^h
 \qquad \text{for $h \in \mathbb{Z}.$}
\end{align*}
We show that the first $H$ Fourier coefficients are small unless $f \cdot n^{-it}$ is of order at most $H$ for some $t \in \mathbb{R}.$
\begin{proposition}\label{prop:mode-annihilation}
Let $f \in \mathcal{M}$ and $H \in \mathbb{N}.$ Suppose $f^h \not\equiv n^{it}$ for each $1 \leq h \leq H$ and $t \in \mathbb{R}.$ 
Then
\begin{align*}
\lim_{K \to \infty}\liminf_{N \to \infty}\max_{1 \leq h \leq H} |\widehat{\sigma}_{K,N}(h)|=0.
\end{align*}

\begin{proof}
Decompose $\{1,\ldots,H\}=\mathcal{H} \cup \mathcal{E},$ where 
\begin{align*}
\mathcal{H}:=\{ 1 \leq h \leq H \,:\, f^h \text{ is pretentious} \},
\end{align*}
and
\begin{align*}
\mathcal{E}:=\{ 1 \leq h \leq H \,:\, f^h \text{ is non-pretentious} \}.
\end{align*}
For $h \in \mathcal{H},$ there exist a primitive character $\chi_h$ of conductor $q_h$ and $t_h \in \mathbb{R}$ such that $f^h \sim \widetilde{\chi}_h \cdot n^{it_h}.$ We shall apply Lemma \ref{lem:bezout-progression} with 
\begin{align*}
q=\prod_{h \in \mathcal H} q_h.
\end{align*}
Since $A=B(B-1),$ Lemma \ref{lem:bezout-progression} (b) gives
\begin{align*}
A(Ln+\ell)+B=BQ^+(R^+n+r^+)
\end{align*}
and
\begin{align*}
A(Ln+\ell)+(B-1)=(B-1)Q^-(R^-n+r^-)
\end{align*}
for every integer $n \geq 0.$ Suppressing the subscript $Q,$ for every $1 \leq h \leq H ,$ we have
\begin{align} \label{eq:zh}
z_{Q,n}^h=f^h \left( \frac{BQ^+}{(B-1)Q^-} \right) f^h (R^+n+r^+)  \overline{f^h (R^-n+r^-)}.
\end{align}

Let $h \in \mathcal{H}$ and $f_h:=f^h \cdot n^{-it_h}$. Then $f_h \sim \widetilde\chi_h$ and
$\D(f^h,\widetilde\chi_h \cdot n^{it_h} ;K,\infty) = \D(f_h,\widetilde\chi_h ;K,\infty)$ for every $K \geq K(q,B).$
Applying Lemma \ref{lem:bezout-progression} (a) and (b), we have
\begin{align*}
 \prod_{p \leq K} p  \mid \operatorname{gcd}(R^+,R^-) \qquad \text{\rm and} \qquad q_h \mid \operatorname{gcd}(R^+,R^-),
\end{align*}
and all prime divisors of $R^+, R^-$ are at most $K$.
Therefore, Lemma \ref{lem:pretcorr} is applicable and gives
\begin{gather} 
\E_{n\leq N}^{\log}
f^h(R^+n+r^+)\overline{f^h(R^-n+r^-)} 
=\chi_h(r^+)\overline{\chi_h(r^-)}(R^+/R^-)^{it_h}
\exp(-2\Re F_h(\chi_h;K,N)) \nonumber\\
+O(
\D(f_h,\widetilde\chi_h ;K,\infty)
+K^{-1/2}) + o_{K,t_h;N \to \infty}(1)\label{eq:corrfh}
\end{gather}
provided that $N>K,$ where
\begin{align*}
F_h(\chi_h;K,N):=\sum_{\substack{K<p \leq N}} \frac{1-f_h(p)\overline{\chi_h}(p)}{p}.
\end{align*}
By Lemma \ref{lem:bezout-progression} (a), (b), and (c) with our choice of $q$, we have 
$\chi_h(r^+)=\chi_h(r^-) \neq 0$ and
\begin{align*}
\frac{R^+}{R^-}=\frac{(B-1)Q^-}{BQ^+}.
\end{align*}
Combining with (\ref{eq:zh}) and (\ref{eq:corrfh}), we obtain
\begin{gather} 
\mathbb{E}^{\log}_{n \leq N} z_{Q,n}^h =  f_h \left(\frac{BQ^+}{(B-1)Q^-} \right) \exp(-2\Re F_h(\chi_h;K,N)) \nonumber \\
+O(\D(f_h,\widetilde\chi_h;K,\infty) \label{eq:elogzh}
+K^{-1/2})+o_{K,t_h;N \to \infty}(1).
\end{gather}
Let $g_h \in \mathcal{M}$ be defined on primes by
\begin{align*}
g_h(p) :=
\begin{cases}
f_h(p) & \mbox{\rm if $p \mid B,$} \\
\overline{f_h(p)}   & \mbox{\rm if $p \nmid B,$}
\end{cases}
\end{align*}
and extended completely multiplicatively. Then
\begin{align} \label{eq:gh}
 \mathbb{E}_{Q \in \Phi_K} f_h\left( \frac{BQ^+}{(B-1)Q^-} \right) 
&= f_h  \left( \frac{B}{B-1} \right) \mathbb{E}_{Q \in \Phi_K} f _h (Q^+) 
\overline{f _h (Q^-)} \nonumber \\
&=  g_h ( B(B-1)) \mathbb{E}_{Q \in \Phi_K} g_h(Q).
\end{align}
Since $g_h \not\equiv 1$ by assumption, there exists a prime $p$ such that $g_h(p) \neq 1.$ If $K>p,$ then Lemma \ref{lem:orthog} gives
\begin{align*}
| \mathbb{E}_{Q \in \Phi_{K}} g_h(Q) | \ll \frac{1}{K|1-g_h(p)|}.
\end{align*}
Combining with (\ref{eq:elogzh}) and (\ref{eq:gh}), we obtain
\begin{align*}
|\mathbb{E}_{Q \in \Phi_{K}}
\mathbb{E}^{\log}_{n \leq N} z_{Q,n}^h | &\ll   \frac{1}{K|1-g_h(p)|} \cdot 
\exp(-2\Re F_h(\chi_h;K,N))
+\D(f_h,\widetilde\chi_h;K,\infty)+K^{-1/2} \\
&\ll \frac{1}{K|1-f_h(p)|}
+\D(f_h,\widetilde\chi_h;K,\infty)+K^{-1/2}+o_{K,t_h;N \to \infty}(1).
\end{align*}
Since $f_h \sim \widetilde{\chi}_h,$ or equivalently $\D(f_h,\widetilde\chi_h;1,\infty)<\infty,$ we have
\begin{align*}
\lim_{K \to \infty} \D(f_h,\widetilde\chi_h;K,\infty) =0.
\end{align*}
Therefore, we conclude that
\begin{align} \label{eq:hinH}
\lim_{K \to \infty}\limsup_{N \to \infty} \max_{h \in \mathcal{H}} |\mathbb{E}_{Q \in \Phi_{K}} 
\mathbb{E}^{\log}_{n \leq N} z_{Q,n}^h |  = 0.
\end{align}

For $h \in \mathcal{E},$ by Lemma \ref{lem:simultaneous-nonpretentiousness}, there exists a strictly increasing sequence of positive integers $(N_j)$ such that
\begin{align} \label{eq:taocond}
 \min_{h\in\mathcal E}
 \inf_{\substack{ \chi \Mod{r},\, r\leq R\\
                  |t|\leq RN_j}}
 \D(f^h,\widetilde{{\chi}} \cdot n^{it};1,N_j)
\geq R
\end{align}
for some suitable $R \geq 1$ to be chosen later, provided that $N_j>R$ is sufficiently large. Let $K \geq K(q,B)$ and $Q \in \Phi_K.$ Applying Lemma \ref{lem:bezout-progression} (d), we have $R_Q^+r_Q^--R_Q^-r_Q^+ = -Q \neq 0,$ and therefore Lemma \ref{lem:two-point-correlation} is applicable. For every $\epsilon>0,$ choose
\begin{align*}
R:=\max_{Q \in \Phi_K} R(\epsilon,R_Q^+,r_Q^+,R_Q^-,r_Q^-).
\end{align*} 
Since (\ref{eq:taocond}) is satisfied, it follows from Lemma \ref{lem:two-point-correlation} that
\[
\max_{Q \in \Phi_K}  \left|
 \E_{n\leq N_j}^{\log}
 f^h(R_Q^+n+r_Q^+)\overline{f^h(R_Q^-n+r_Q^-)}
 \right|\leq\epsilon.
\]
Using (\ref{eq:zh}), we obtain
\begin{align*}
|\mathbb{E}_{Q \in \Phi_{K}}
\mathbb{E}^{\log}_{n \leq N_j} z_{Q,n}^h | \leq \epsilon.
\end{align*}
Since $\epsilon>0$ can be chosen arbitrarily small, we conclude that
\begin{align} \label{eq:hinE}
\lim_{K \to \infty}\liminf_{N \to \infty} \max_{h \in \mathcal{E}}
|\mathbb{E}_{Q \in \Phi_{K}}
\mathbb{E}^{\log}_{n \leq N_j} z_{Q,n}^h | =0.
\end{align}
Finally, recall that $\widehat{\sigma}_{K,N}(h)=\mathbb{E}_{Q \in \Phi_{K}}
\mathbb{E}^{\log}_{n \leq N_j} z_{Q,n}^h$ for $h \in \mathbb{Z}.$ Combining (\ref{eq:hinH}) and (\ref{eq:hinE}), we arrive at
\begin{align*}
\lim_{K \to \infty}\liminf_{N \to \infty}\max_{1 \leq h \leq H} |\widehat{\sigma}_{K,N}(h)|
= 0,
\end{align*}
and the proof is complete.
\end{proof}

\end{proposition}

\section{Proof of Theorem \ref{thm:main}: Infinite order}

Applying the Erdős--Turán inequality, we prove Theorem \ref{thm:main} in the case where 
\(f \cdot n^{-it}\) has infinite order for every \(t \in \mathbb{R}\).

\begin{lemma}[Erdős--Turán inequality] \label{lem:erdos-turan}

Let $\nu$ be a probability measure on $\T$, and
denote its Fourier coefficients by
\[
 \widehat\nu(h):=\int_{\T}z^h\,d\nu(z) \qquad \text{for $h\in\mathbb Z.$ }
\]
Also, denote by $\mathcal{I}$ the collection of half-open arcs $I\subseteq\T$. Then for every $H \in \N$, we have
\[
 \sup_{I \in \mathcal{I}}
 \left|\nu(I)-\lambda(I)\right|
 \ll
 \frac1H+\sum_{h=1}^{H}\frac{|\widehat\nu(h)|}{h},
\]
where $\lambda$ denotes the normalized Lebesgue (Haar) measure on $\mathbb{T}$.
\end{lemma}

\begin{proof}
See \cite[Section 2]{MR1218222}, and also \cite[Lemma 2.3]{MR3856825}.
\end{proof}

Applying Lemma \ref{lem:normalization} and then Lemma \ref{lem:reduction}, it suffices to verify the special case.

\begin{proposition} \label{prop:inf}
Let $f \in \mathcal{M}$ and $A,B\in\N$ with $B\geq 2$ and $A=B(B-1).$ Suppose $f \cdot n^{-it} $ has infinite order for every $t \in \mathbb{R}.$ Then
\begin{align*}
\omega \left(  f \left( \frac{Am+B}{Am+(B-1)} \right)  \right)= \mathbb{T}.
\end{align*}

\begin{proof}
It suffices to show that 
\begin{align*}
\lim_{K \to \infty}\liminf_{N \to \infty} \sup_{I \in \mathcal{I}}
 \left|{\sigma}_{K,N}(I)-\lambda(I)\right| = 0.
\end{align*}
Applying Lemma \ref{lem:erdos-turan} with $\nu={\sigma}_{K,N},$ for every $H \in \mathbb{N},$ we have
\begin{align*}
 \sup_{I \in \mathcal{I}}
 \left|{\sigma}_{K,N}(I)-\lambda(I)\right| \ll 
  \frac1H+\sum_{h=1}^{H}\frac{|\widehat{\sigma}_{K,N}(h)|}{h}.
\end{align*}
Since the assumption  $f \cdot n^{-it} $ has infinite order for every $t \in \mathbb{R}$ is equivalent to $f^h \not\equiv n^{it}$ for every $h \in \mathbb{N}$ and $t \in \mathbb{R}$, Proposition \ref{prop:mode-annihilation} is applicable and gives 
\begin{align*}
\lim_{K \to \infty}\liminf_{N \to \infty} \sup_{I \in \mathcal{I}}
 \left|{\sigma}_{K,N}(I)-\lambda(I)\right|  \ll \frac{1}{H}.
\end{align*}
Since $H \in \N$ can be chosen arbitrarily large, the proposition follows.
\end{proof}

\end{proposition}


\section{Proof of Theorem \ref{thm:main}: Finite order}

Applying Fourier inversion on finite cyclic groups, we prove Theorem \ref{thm:main} in the case where 
\(f \cdot n^{-it}\) has finite order for some \(t \in \mathbb{R}\).

\begin{lemma}[Fourier inversion on finite cyclic groups]
\label{lem:finite-fourier}
Given $k \in \mathbb{N},$ let $\nu$ be a probability measure supported on $\mu_k$, and denote
\begin{align*}
\widehat\nu(h):=\int_{\mu_k}z^h\,d\nu(z) \qquad \text{for $h\in\mathbb Z.$ }
\end{align*}
Then for every $\zeta\in\mu_k$, we have
\[
 \nu(\{\zeta\})
 =\frac1k\sum_{h=0}^{k-1}
     \widehat\nu(h) \zeta^{-h}.
\]

\begin{proof}
The lemma follows immediately by integrating the orthogonality relation on $\mu_k$ against $\nu.$
\end{proof}
\end{lemma}




Applying Lemma \ref{lem:normalization} and then Lemma \ref{lem:reduction}, it suffices to verify the special case.

\begin{proposition} \label{prop:fin}
Let $f \in \mathcal{M}$ and $A,B\in\N$ with $B\geq 2$ and $A=B(B-1).$ Suppose $f \cdot n^{-it} $ has order $k \in \mathbb{N}$ for some $t \in \mathbb{R}.$ Then
\begin{align*}
\omega \left(  f \left( \frac{Am+B}{Am+(B-1)} \right)  \right)= \mu_k.
\end{align*}
\begin{proof}
Since
\[ \lim_{m \to \infty}
\left(\frac{Am+B}{Am+(B-1)}\right)^{it} = 1,
\]
replacing \(f\) by \(f\cdot n^{-it}\) does not change
the set of limit points. We may therefore assume that \(f\) has order
\(k\). As the case \(k=1\) is clear, we assume \(k\geq 2\).

Since \(\sigma_{K,N}\) is supported on \(\mu_k\), it suffices to show that
\begin{align*}
\lim_{K\to\infty}\liminf_{N\to\infty}
\max_{\zeta\in\mu_k}
\left|
\sigma_{K,N}(\{\zeta\})-\frac1k
\right|=0.
\end{align*}
Applying Lemma \ref{lem:finite-fourier} with
\(\nu=\sigma_{K,N}\), for every \(\zeta\in\mu_k\), we have
\begin{align*}
\left|
\sigma_{K,N}(\{\zeta\})-\frac1k
\right|
\leq
\frac1k\sum_{h=1}^{k-1}
\left|\widehat{\sigma}_{K,N}(h)\right|.
\end{align*}
Since \(f\) has order \(k\), we have
\(f^h\not\equiv n^{it}\) for every \(1\leq h\leq k-1\) and
\(t\in\mathbb{R}\). Therefore, Proposition \ref{prop:mode-annihilation} is 
applicable, so that
\begin{align*}
\lim_{K\to\infty}\liminf_{N\to\infty}
\max_{\zeta\in\mu_k}
\left|
\sigma_{K,N}(\{\zeta\})-\frac1k
\right|=0,
\end{align*}
and the proposition follows.
\end{proof}

\end{proposition}


\section*{Acknowledgements}
The author would like to thank Oleksiy Klurman and Sacha Mangerel for their friendly support.

\clearpage

\printbibliography

@article {charamaras2024multiplicativerecurrencelinearpatterns,
    AUTHOR = {Charamaras, D. and Mountakis, A. and Tsinas,
              K.},
     TITLE = {On multiplicative recurrence along linear patterns},
   JOURNAL = {J. Lond. Math. Soc. (2)},
  FJOURNAL = {Journal of the London Mathematical Society. Second Series},
    VOLUME = {112},
      YEAR = {2025},
    NUMBER = {3},
     PAGES = {Paper No. e70292, 53},
  %    ISSN = {0024-6107,1469-7750},
 %  MRCLASS = {11N37 (37A44 37B20)},
%%  MRNUMBER = {4960398},
%       DOI = {10.1112/jlms.70292},
 %      URL = {https://doi.org/10.1112/jlms.70292},
}

@article {MR3856825,
    AUTHOR = {Klurman, O. and Mangerel, A. P.},
     TITLE = {Rigidity theorems for multiplicative functions},
   JOURNAL = {Math. Ann.},
  FJOURNAL = {Mathematische Annalen},
    VOLUME = {372},
      YEAR = {2018},
    NUMBER = {1-2},
     PAGES = {651--697},
   %   ISSN = {0025-5831,1432-1807},
  % MRCLASS = {11N37},
 % MRNUMBER = {3856825},
%MRREVIEWER = {Y.-F.\ S.\ P\'etermann},
     %  DOI = {10.1007/s00208-018-1724-6},
    %   URL = {https://doi.org/10.1007/s00208-018-1724-%6},
}

@article {MR4594405,
    AUTHOR = {Donoso, S. and Le, A. N. and Moreira, J. and Sun,
              W.},
     TITLE = {Additive averages of multiplicative correlation sequences and
              applications},
   JOURNAL = {J. Anal. Math.},
  FJOURNAL = {Journal d'Analyse Math\'ematique},
    VOLUME = {149},
      YEAR = {2023},
    NUMBER = {2},
     PAGES = {719--761},
 %     ISSN = {0021-7670,1565-8538},
%   MRCLASS = {37P30 (11B30 11N37)},
 % MRNUMBER = {4594405},
%       DOI = {10.1007/s11854-022-0264-x},
%       URL = {https://doi.org/10.1007/s11854-022-0264-x},
}

@misc{leung2025multiplicativerecurrencemobiustransformations,
      title={Multiplicative recurrence of M\"obius transformations}, 
      author={S.-K. Leung and C. Táfula},
      year={2025},
      eprint={2409.12936},
      archivePrefix={arXiv},
%      primaryClass={math.NT},
  %    url={https://arxiv.org/abs/2409.12936}, 
}

@article {MR4328688,
    AUTHOR = {Klurman, O. and Mangerel, A. P. and Pohoata,
              C. and Ter\"av\"ainen, J.},
     TITLE = {Multiplicative functions that are close to their mean},
   JOURNAL = {Trans. Amer. Math. Soc.},
  FJOURNAL = {Transactions of the American Mathematical Society},
    VOLUME = {374},
      YEAR = {2021},
    NUMBER = {11},
     PAGES = {7967--7990},
   %   ISSN = {0002-9947,1088-6850},
 %  MRCLASS = {11N37 (11N64)},
%  MRNUMBER = {4328688},
%MRREVIEWER = {Peter\ Shiu},
  %     DOI = {10.1090/tran/8427},
 %      URL = {https://doi.org/10.1090/tran/8427},
}

@article {MR3569059,
    AUTHOR = {Tao, T},
     TITLE = {The logarithmically averaged {C}howla and {E}lliott
              conjectures for two-point correlations},
   JOURNAL = {Forum Math. Pi},
  FJOURNAL = {Forum of Mathematics. Pi},
    VOLUME = {4},
      YEAR = {2016},
     PAGES = {e8, 36},
 %     ISSN = {2050-5086},
%   MRCLASS = {11N37},
%  MRNUMBER = {3569059},
%MRREVIEWER = {Y.-F.\ S.\ P\'etermann},
 %      DOI = {10.1017/fmp.2016.6},
%       URL = {https://doi.org/10.1017/fmp.2016.6},
}

@incollection {MR1218222,
    AUTHOR = {Ruzsa, I. Z.},
     TITLE = {On an inequality of {E}rd\H os and {T}ur\'an concerning
              uniform distribution modulo one. {I}},
 BOOKTITLE = {Sets, graphs and numbers ({B}udapest, 1991)},
    SERIES = {Colloq. Math. Soc. J\'anos Bolyai},
    VOLUME = {60},
     PAGES = {621--630},
 PUBLISHER = {North-Holland, Amsterdam},
      YEAR = {1992},
 %     ISBN = {0-444-98681-2},
 %  MRCLASS = {11K38 (11K06)},
 % MRNUMBER = {1218222},
%MRREVIEWER = {Gerhard\ Larcher},
}

@article {frantzikinakis2023partitionregularitypythagoreanpairs,
    AUTHOR = {Frantzikinakis, N. and Klurman, O. and Moreira, J.},
     TITLE = {Partition regularity of {P}ythagorean pairs},
   JOURNAL = {Forum Math. Pi},
  FJOURNAL = {Forum of Mathematics. Pi},
    VOLUME = {13},
      YEAR = {2025},
     PAGES = {Paper No. e5, 52},
  %    ISSN = {2050-5086},
 %  MRCLASS = {05D10 (11B30 11N37 37A44 37H15)},
 % MRNUMBER = {4862834},
 %      DOI = {10.1017/fmp.2024.27},
 %      URL = {https://doi.org/10.1017/fmp.2024.27},
}

@book{gs,
  title={Multiplicative number theory: The pretentious approach},
  author={Granville, A. and Soundararajan, K.},
url={https://dms.umontreal.ca/~andrew/PDF/Book.To2.5.pdf},
year={2014},
note={Book manuscript in preparation}
}

@article {MR3533300,
    AUTHOR = {Tao, T.},
     TITLE = {The {E}rd\H os discrepancy problem},
   JOURNAL = {Discrete Anal.},
  FJOURNAL = {Discrete Analysis},
      YEAR = {2016},
     PAGES = {Paper No. 1, 29},
  %    ISSN = {2397-3129},
%   MRCLASS = {11K38},
%  MRNUMBER = {3533300},
%MRREVIEWER = {Friedrich\ Pillichshammer},
 %      DOI = {10.19086/da.609},
%       URL = {https://doi.org/10.19086/da.609},
}

@misc{frantzikinakis2024partitionregularitygeneralizedpythagorean,
      title={Partition regularity of generalized Pythagorean pairs}, 
      author={N. Frantzikinakis and O. Klurman and J. Moreira},
      year={2024},
      eprint={2407.08360},
      archivePrefix={arXiv},
%      primaryClass={math.CO},
 %     url={https://arxiv.org/abs/2407.08360}, 
}


\end{document}